\documentclass[11pt]{article}

\usepackage{amssymb, amsmath, amsthm, graphicx}
\usepackage[left=1in,top=1in,right=1in]{geometry}

\date{}

\theoremstyle{plain}
      \newtheorem{theorem}{Theorem}[section]
      \newtheorem{lemma}[theorem]{Lemma}
            \newtheorem{claim}[theorem]{Claim}

      \newtheorem{conjecture}[theorem]{Conjecture}
\theoremstyle{definition}
      \newtheorem{definition}[theorem]{Definition}

\theoremstyle{remark}

\def\twr{\mbox{\rm twr}}

\title{Constructions in Ramsey theory}

\author{Dhruv Mubayi\thanks{Department of Mathematics, Statistics, and Computer Science, University of Illinois, Chicago, IL, 60607 USA.  Research partially supported by NSF grant DMS-1300138. Email: {\tt mubayi@uic.edu}} \and Andrew Suk\thanks{Department of Mathematics,  University of California at San Diego, La Jolla, CA, 92093 USA. Supported by NSF grant DMS-1500153, an NSF CAREER award, and an Alfred Sloan Fellowship. Email: {\tt asuk@ucsd.edu}\newline
}}

\begin{document}

\maketitle

\begin{abstract}
We provide several constructions for problems in Ramsey theory. First, we prove a superexponential lower bound for the classical 4-uniform Ramsey number $r_4(5,n)$, and  the same for the iterated $(k-4)$-fold logarithm of the $k$-uniform version $r_k(k+1,n)$. This is the first improvement of the original exponential lower bound for $r_4(5,n)$ implicit in work of Erd\H os and Hajnal from 1972 and also improves the current best known bounds for larger $k$ due to the authors. Second, we prove an upper bound for the hypergraph Erd\H os-Rogers function $f^k_{k+1, k+2}(N)$ that is an iterated $(k-13)$-fold logarithm in $N$. This improves the previous upper bounds that were only logarithmic and addresses a question of Dudek and the first author that was reiterated by Conlon, Fox and Sudakov.  Third, we generalize the results of Erd\H os and Hajnal about the 3-uniform Ramsey number of $K_4$ minus an edge versus a clique to $k$-uniform hypergraphs.
\end{abstract}

\section{Introduction}

A $k$-uniform hypergraph $H$ ($k$-graph for short) with vertex set $V$ is a collection of $k$-element subsets of $V$.  We write $K^k_n$ for the complete $k$-uniform hypergraph on an $n$-element vertex set.  Given $k$-graphs $F$, $G$, the \emph{Ramsey number} $r(F, G)$ is the minimum $N$ such that every red/blue coloring of the edges of $K^k_N$ results in a monochromatic red copy of $F$ or a monochromatic blue copy of $G$.

In this paper, we study several problems in hypergraph Ramsey theory.  We describe each problem in detail in its relevant section. Here we provide a brief summary.  In Section \ref{clique}, we give new lower bounds on the classical Ramsey number $r(K^k_{k + 1},K^k_{n})$, improving the previous best known bounds obtained by the authors \cite{MS15}.  In particular, we give the first superexponential lower bound for $r(K^4_{5},K^4_{n})$ since the problem was first explicitly stated by Erd\H os and Hajnal~\cite{EH72} in 1972. In Section \ref{rogers}, we establish a new upper bound for the hypergraph Erd\H os-Rogers function $f^k_{k+1, k+2}(N)$ that is an iterated logarithm function in $N$. More precisely, we construct $k$-graphs on $N$ vertices, with no copy of $K_{k+2}^k$, yet every set of $n$ vertices contains a copy of $K_{k+1}^k$ where $n$ is the $(k-13)$-fold iterated logarithm of $N$.
This addresses questions posed by Dudek and the first author~\cite{DM} as well as by Conlon, Fox, and Sudakov~\cite{CFS14} and
significantly improves the previous best known bound in~\cite{DM} of $n=O((\log N)^{1/(k-1)})$.
In Section \ref{last} we study the Ramsey numbers for $k$-half-graphs versus cliques, generalizing the results of Erd\H os and Hajnal~\cite{EH72} about the 3-uniform Ramsey number of $K_4$ minus an edge versus a clique.  The upper bound is a straightforward extension of the method in~\cite{EH72}, while the constructions are new.

All logarithms are base 2 unless otherwise stated. For the sake of clarity of presentation, we systematically omit floor and ceiling signs whenever they are not crucial.

\section{A new lower bound for $r_k(k + 1,n)$}\label{clique}

 In order to avoid the excessive use of superscripts, we use the simpler notation $r(K^k_s, K^k_n) = r_k(s,n)$.  Estimating the Ramsey number $r_k(s,n)$ is a classical problem in extremal combinatorics and has been extensively studied \cite{EHR,ER,ES35}.  Here we study the \emph{off-diagonal} Ramsey number, that is,  $r_k(s,n)$ with $k,s$ fixed and $n$ tending to infinity.  It is known that for fixed $s \geq k + 1$, $r_2(s,n)$ grows polynomially in $n$ \cite{AKS,B,BK} and $r_3(s,n)$ grows exponentially in a power of $n$ \cite{CFS}.  In 1972, Erd\H os and Hajnal~\cite{EH72} raised the question of determining the correct tower growth rate for $r_k(s ,n)$. We define the \emph{tower function} $\twr_k(x)$ by
 $$\twr_1(x) = x \qquad \hbox{ and } \qquad \twr_{i + 1} = 2^{\twr_i(x)}.$$

By applying the Erd\H os-Hajnal stepping up lemma in the off-diagonal setting (see \cite{graham}), it follows that $r_k(s,n) \geq \twr_{k-1}(\Omega(n))$, for $k\geq 4$ and for all $s \geq 2^{k-1} - k + 3$.   However they conjectured the following.

\begin{conjecture}{\bf (Erd\H os-Hajnal~\cite{EH72})}\label{conj2}
For $s \geq k + 1 \ge 5$ fixed,  $r_k(s,n) \geq\twr_{k-1}(\Omega(n))$.

\end{conjecture}

\noindent In \cite{CFS13}, Conlon, Fox, and Sudakov modified the Erd\H os-Hajnal stepping-up lemma to show that Conjecture \ref{conj2} holds for all $s\geq \lceil 5k/2\rceil - 3$.   Recently the authors nearly proved the conjecture by establishing the following.

\begin{theorem}[\cite{MS15}] \label{main0}
There is a positive constant $c>0$ such that the following holds.  For $k \geq 4$ and $n>3k$,  we have
\begin{enumerate}
\item $r_k(k + 3,n) \geq \twr_{k-1}(cn),$

\item $r_k(k + 2,n) \geq \twr_{k - 1}(c\log^2 n),$

\item $r_k(k + 1, n) \geq \twr_{k - 2}(cn^2).$

\end{enumerate}

\end{theorem}

Implicit in work of Erd\H os and Hajnal~\cite{EH72} is the bound
$r_4(5,n) > 2^{cn}$ for some absolute positive constant $c$. While the authors~\cite{MS15} recently improved this to $2^{cn^2}$ above, there has been no superexponential lower bound given for this basic problem.  Here we provide such a lower bound.

\begin{theorem}\label{mainthm}
There is an absolute constant $c>0$ such that
$$r_4(5,n)> 2^{n^{c\log\log n}},$$
 and more generally for $k > 4$,
$$r_k(k + 1,n) > \twr_{k - 2}(n^{c\log \log n}).$$

\end{theorem}

One of the building blocks we will use in our construction is the following lower bound of Conlon, Fox, and Sudakov \cite{CFS}: there is an absolute positive constant $c>0$ such that
\begin{equation} \label{r34t} r_3(4,t) > 2^{c \, t\log t}.\end{equation}
Our lower bound for $r_4(5,n)$ is proved via the following theorem.

\begin{theorem}\label{stepup4}
For $n$ sufficiently large, we have $$r_4(5, n) > 2^{r_3(4, \lfloor(\log n)/2\rfloor) - 1}.$$

\end{theorem}

\medskip

 \proof The idea is to apply a variant of the Erd\H os-Hajnal stepping up lemma (see \cite{graham}).  Set $t = \lfloor\frac{\log n}{2}\rfloor$.  Let $\phi$ be a red/blue coloring of the edges of the complete $3$-uniform hypergraph on the vertex set $\{0,1,\ldots, r_3\left(4,t\right) - 2\}$ without a red $K_{4}^{3}$ and without a blue $K_{t}^{3}$. We use $\phi$ to define a red/blue coloring $\chi$ of the edges of the complete $4$-uniform hypergraph $K_N^{4}$ on the vertex set $V=\{0,1,\ldots, N-1\}$ with $N= 2^{r_3(4, t) - 1}$, as follows.

For any $a \in V$, write $a=\sum_{i=0}^{r_3(4, t)-2}a(i)2^i$ with $a(i) \in \{0,1\}$ for each $i$. For $a \not = b$, let $\delta(a,b)$ denote the largest $i$ for which $a(i) \not = b(i)$.  Notice that we have the following stepping-up properties (again see \cite{graham})

\begin{description}

\item[Property A:] For every triple $a < b < c$, $\delta(a,b) \not = \delta(b,c)$ .

\item[Property B:] For $a_1 < \cdots < a_r$, $\delta(a_1,a_{r}) = \max_{1 \leq j \leq r-1}\delta(a_j,a_{j + 1})$.

\end{description}

Given any $4$-tuple $a_1< \cdots <a_4$ of $V$, consider the integers $\delta_i=\delta(a_i,a_{i+1}), 1\le i\le 3$. Say that $\delta_1,\delta_2,\delta_{3}$ forms a monotone sequence if $\delta_1 < \delta_2 < \delta_3$ or $\delta_1 > \delta_2 > \delta_3$. Now, define $\chi$ as follows:

\begin{equation}
 \chi(a_1, a_2, a_3, a_4) =\begin{cases}

\phi(\delta_1,\delta_2,\delta_{3}) \qquad \hbox{ if $\delta_1,\delta_2,\delta_{3}$ is monotone} \cr
blue \qquad \qquad \quad \hbox{\, if $\delta_1,\delta_2,\delta_{3}$ is not monotone} \notag
\end{cases}
\end{equation}

 Hence we have the following property which can be easily verified using Properties A and B (see \cite{graham}).

\begin{description}

\item[Property C:] For $a_1 < \cdots < a_r$, set $\delta_j = \delta(a_j,a_{j + 1})$ and suppose that $\delta_1,\ldots, \delta_{r-1}$ form a monotone sequence.  If $\chi$ colors every 4-tuple in $\{a_1,\ldots, a_r\}$ red (blue), then $\phi$ colors every triple in $\{\delta_1,\ldots, \delta_{r-1}\}$ red (blue).

\end{description}

For sake of contradiction, suppose that the coloring $\chi$ produces a red $K_5^{4}$ on vertices $a_1 < \cdots < a_5$, and let $\delta_i = \delta(a_i,a_{i + 1})$, $1 \leq i \leq 4$.  Then $\delta_1,\ldots, \delta_4$ form a monotone sequence and, by Property C, $\phi$ colors every triple in $\{\delta_1,\ldots, \delta_4\}$ red which is a contradiction.  Therefore, there is no red $K_5^{4}$ in coloring $\chi$.

Next we show that there is no blue $K_n^{4}$ in coloring $\chi$. Our argument is reminiscent of the standard argument for the bound $r_2(n,n) < 4^n$, though it must be adapted to this setting.   For sake of contradiction, suppose we have vertices $a_1,\ldots, a_n \in V$ such that $a_1 < \cdots < a_n$ and $\chi$ colors every $4$-tuple in the set $\{a_1, \ldots, a_n\}$ blue.  Let $\delta_i = \delta(a_i,a_{i + 1})$ for $1\leq i \leq n - 1$.  We greedily construct a set $D_h = \{\delta_{i_1},\ldots, \delta_{i_h}\}\subset \{\delta_1,\ldots, \delta_{n - 1}\}$ and a set $S_h \subset \{a_1,\ldots, a_n\}$ such that the following holds.

\begin{enumerate}
\item We have $\delta_{i_1} > \cdots > \delta_{i_h}$.

\item For each $\delta_{i_j}  = \delta(a_{i_j},a_{i_j + 1}) \in D_h =\{\delta_{i_1},\ldots, \delta_{i_h}\}$, consider the set of vertices $$A =  \{a_{i_{j + 1}},a_{i_{j + 1} + 1},\ldots, a_{i_{h}}, a_{i_h + 1}\}\cup S_h.$$  Then either every element in $A$ is greater than $a_{i_j}$ or every element in $A$ is less than $a_{i_j + 1}$.  In the former case we will label $\delta_{i_j}$ \emph{white}, in the latter case we label it \emph{black}.

\item The indices of the vertices in $S_h$ are consecutive, that is, $S_h = \{a_r,a_{r + 1},\ldots, a_{s-1},a_s\}$ for $1 \leq r < s \leq n$.

\end{enumerate}

We start with the $D_0 = \emptyset$ and $S_0 = \{a_1,\ldots, a_{n}\}$.  Having obtained $D_h = \{\delta_{i_1},\ldots, \delta_{i_h}\}$ and $S_h = \{a_r,\ldots, a_s\}$, $1\leq r < s \leq n$, we construct $D_{h  + 1}$ and $S_{h + 1}$ as follows.  Let $\delta_{i_{h + 1}}  = \delta(a_{\ell},a_{\ell + 1})$ be the unique largest element in $\{\delta_r,\delta_{r + 1},\ldots, \delta_{s-1}\}$, and set $D_{h + 1} = D_h\cup \delta_{i_{h + 1}}$.  The uniqueness of $\delta_{i_{h + 1}}$ follows from Properties A and B.  If $|\{a_r, a_{r+1},\ldots, a_{\ell}\}|\geq |S_h|/2$, then we set $S_{h  + 1} = \{a_r, a_{r+1},\ldots, a_{\ell}\}$.  Otherwise by the pigeonhole principle, we have $|\{a_{\ell + 1},a_{\ell + 2}, \ldots, a_s\}| \geq |S_h|/2$ and we set $S_{h + 1} = \{a_{\ell + 1},a_{\ell + 2}, \ldots, a_s\}$.

Since $|S_0| = n$, $t = \lfloor\frac{\log n}{2}\rfloor$ and $|S_{h  + 1}| \geq |S_{h}|/2$ for $h \geq 0$, we can construct $D_{2t} = \{\delta_{i_1},\ldots, \delta_{i_{2t}}\}$ with the desired properties.  By the pigeonhole principle, at least $t$ elements in $D_{2t}$ have the same label, say \emph{white}.  The other case will follow by a symmetric argument.  We remove all black labeled elements in $D_{2t}$, and let $\{\delta_{j_1},\ldots, \delta_{j_t}\}$ be the resulting set.  Now consider the vertices $a_{j_1}, a_{j_2}, \ldots, a_{j_{t}} \in V$.  By construction and by Property B, we have $a_{j_1} < a_{j_2} < \cdots < a_{j_{t}}$ and $\delta(a_{j_1},a_{j_2}) = \delta_{i_{j_1}}, \delta(a_{j_2},a_{j_3}) = \delta_{i_{j_2}}, \ldots, \delta(a_{j_{t }}, a_{j_{t +1}}) = \delta_{i_{j_{t}}} $.  Therefore we have a monotone sequence $$\delta(a_{j_1},a_{j_2}) > \delta(a_{j_2},a_{j_3}) > \cdots >  \delta(a_{j_{t }}, a_{j_{t +1}}).$$

By Property C, $\phi$ colors every triple from this set blue which is a contradiction.  Therefore there is no red $K_5^{4}$ and no blue $K_n^{4}$ in coloring $\chi$. \qed

Applying the lower bound in (\ref{r34t}), we obtain that
$$r_4(5,n) \geq 2^{r_3(4, \lfloor\log n/2\rfloor) - 1} > 2^{2^{c \log n \log\log n}}= 2^{n^{c \log\log n}}$$
 for some absolute positive constant $c$ and this establishes the first part of Theorem \ref{mainthm}.

  We next  prove Theorem \ref{mainthm} for $k\geq 5$.  Independently, Conlon, Fox and Sudakov~\cite{CFS15} gave a different proof of Theorem \ref{main0} part 1.  Their approach was to begin with a known  4-uniform construction that yields $r_4(7,n) > 2^{2^{cn}}$ and then use a variant of the stepping up lemma to give tower-type lower bounds for larger $k$. Unfortunately, this variant of the stepping up lemma does not work if one begins instead with a lower bound for
$r_4(5,n)$ which is our case.  However, a further variant of the approach does work, and this is what we do below.

\begin{lemma}\label{stepup4a}

For $k\geq 5$ and $n$ sufficiently large, we have

$$r_{k}(k +1,n) \geq 2^{r_{k-1}(k,\lfloor n/6\rfloor ) - 1}.$$

\end{lemma}

\proof Again we apply a variant of the stepping-up lemma.  Let $\phi$ be a red/blue coloring of the edges of the complete $(k-1)$-uniform hypergraph on the vertex set $\{0,1,\ldots, r_{k-1}(k ,\lfloor n/6\rfloor) - 2\}$ without a red $K_{k}^{k-1}$ and without a blue $K_{\lfloor n/6\rfloor}^{k-1}$. We use $\phi$ to define a red/blue coloring $\chi$ of the edges of the complete $k$-uniform hypergraph $K^k_N$ on the vertex set $V=\{0,1,\ldots, N - 1\}$ with $N= 2^{r_{k-1}(k,\lfloor n/6\rfloor) - 1}$, as follows.

Just as above, for any $a \in V$, write $a=\sum_{i=0}^{r_{k-1}(k,\lfloor n/6\rfloor)-2}a(i)2^i$ with $a(i) \in \{0,1\}$ for each $i$. For $a \not = b$, let $\delta(a,b)$ denote the largest $i$ for which $a(i) \not = b(i)$.  Hence Properties A and B hold.

Given any $k$-tuple $a_1<a_2<\ldots<a_{k}$ of $V$, consider the integers $\delta_i=\delta(a_i,a_{i+1}), 1\le i\le k-1$. We say that $\delta_i$ is a {\it local minimum} if $\delta_{i-1}>\delta_i<\delta_{i+1}$, a {\it local maximum} if $\delta_{i-1}<\delta_i>\delta_{i+1}$, and a {\it local extremum} if it is either a local minimum or a local maximum.  We say that $\delta_i$ is \emph{locally monotone} if $\delta_{i-1} < \delta_i < \delta_{i + 1}$ or $\delta_{i-1} > \delta_i > \delta_{i + 1}$. Since $\delta_{i-1} \not = \delta_i$ for every $i$, every nonmonotone sequence $\delta_1,\ldots,\delta_{k-1}$ has a local extremum.  If $\delta_1,\ldots,\delta_{k-1 }$ form a monotone sequence, then let $\chi(a_1,a_2,\ldots,a_{k})=\phi(\delta_1,\delta_2,\ldots,\delta_{k-1})$.  Otherwise if $\delta_1,\ldots,\delta_{k-1}$ is not monotone, then let $\chi(a_1,a_2,\ldots,a_{k})$ be red if and only if $\delta_2$ is a local maximum and $\delta_3$ is a local minimum.  Hence the following generalization of Property C holds.

\begin{description}

\item[Property D:] For $a_1 < \cdots < a_r$, set $\delta_j = \delta(a_j,a_{j + 1})$ and suppose that $\delta_1,\ldots, \delta_{r-1}$ form a monotone sequence.  If $\chi$ colors every $k$-tuple in $\{a_1,\ldots, a_r\}$ red (blue), then $\phi$ colors every $(k-1)$-tuple in $\{\delta_1,\ldots, \delta_{r-1}\}$ red (blue).

\end{description}

For sake of contradiction, suppose that the coloring $\chi$ produces a red $K_{k + 1}^k$ on vertices $a_1 < \cdots < a_{k + 1}$, and let $\delta_i = \delta(a_i,a_{i + 1})$, $1 \leq i \leq k$.  We have two cases.

\emph{Case 1.}  Suppose $\delta_1,\ldots, \delta_{k-1}$ is monotone.  Then if $\delta_2,\ldots, \delta_{k}$ is also a monotone sequence, $\phi$ colors every $(k-1)$-tuple in $\{\delta_1,\ldots, \delta_{k}\}$ red by Property D, which is a contradiction. Otherwise, $\delta_{k-1}$ is the only local extremum and $\chi(a_2,\ldots, a_{k + 1})$ is blue, which is again a contradiction.

\emph{Case 2.} Suppose $\delta_1,\ldots, \delta_{k-1}$ is not monotone.  Then we know that $\delta_2$ is a local maximum and $\delta_3$ is a local minimum.  However this implies that $\chi(a_2,\ldots, a_{k + 1})$ is blue, which is a contradiction.  Hence there is no red $K_{k + 1}^k$ in coloring $\chi$.

Next we show that there is no blue $K_n^k$ in coloring $\chi$.  For sake of contradiction, suppose we have vertices $a_1,\ldots, a_n \in V$ such that $a_1 < \cdots < a_n$ and $\chi$ colors every $k$-tuple blue, and let $\delta_i = \delta(a_i,a_{i + 1})$ for $1\leq i \leq n - 1$.  By Property D, there is no integer $r$ such that $\delta_r, \delta_{r + 1},\ldots, \delta_{r + \lfloor n/6\rfloor}$ is monotone, since this implies that $\phi$ colors every $(k-1)$-tuple in the set $\{\delta_r, \delta_{r + 1},\ldots, \delta_{r + \lfloor n/6\rfloor}\}$ blue which is a contradiction.  Therefore the sequence $\delta_1,\ldots, \delta_{n-1}$ contains at least four local extrema.  Let $\delta_{j_1}$ be the first local maximum, and let $\delta_{j_2}$ be the next local extremum, which must be a local minimum.  Recall that $\delta_{j_1} = \delta(a_{j_1},a_{j_1 + 1})$ and $\delta_{j_2} = \delta(a_{j_2},a_{j_2 + 1})$.  Consider the $k$ vertices $$a_{j_1 - 1}, a_{j_1}, a_{j_2}, a_{j_2 + 1},a_{j_2 + 2}, \ldots, a_{j_2 + k - 3}$$ and the sequence

$$\delta(a_{j_1 - 1}, a_{j_1}), \delta(a_{j_1}, a_{j_2}), \delta(a_{j_2}, a_{j_2 + 1}), \ldots , \delta( a_{j_2 + k - 4}, a_{j_2 + k - 3}).$$

By Property B we have $\delta(a_{j_1},a_{j_2}) = \delta_{j_1}$, and therefore $\delta(a_{j_1},a_{j_2})$ is a local maximum and $\delta(a_{j_2},a_{j_2 + 1})$ is a local minimum.  Therefore $\chi(a_{j_1 - 1}, a_{j_1}, a_{j_2}, a_{j_2 + 1}, \ldots, a_{j_2 + k - 3})$ is red and we have our contradiction.  Hence there is no blue $K_n^k$ in coloring $\chi$. \qed

By combining Theorem \ref{stepup4} with Lemma \ref{stepup4a}, we establish Theorem \ref{mainthm}.

\section{The Erd\H os-Rogers function for hypergraphs}\label{rogers}
An $s$-independent set in a $k$-graph $H$ is a vertex subset that contains no copy of $K_s^k$. So if $s=k$, then it is just an independent set.  Let $\alpha_s(H)$ denote the size of the largest $s$-independent set in $H$.

\begin{definition}
 For $k \le s < t < N$, the Erd\H os-Rogers function $f^k_{s,t}(N)$ is the minimum of $\alpha_s(H)$ taken over all $K_t^k$-free $k$-graphs $H$ of order
 $N$.
\end{definition}

 To prove the lower bound $f_{s,t}^{k}(N)\ge n$ one must show that every $K_{t}^{k}$-free $k$-graph of order $N$ contains an $s$-independent set with
 $n$ vertices. On the other hand, to prove the upper bound $f_{s,t}^{k}(N) < n$, one must construct a $K_{t}^{k}$-free $k$-graph $H$ of order $N$ with
 $\alpha_s(H) < n$.

The problem of determining $f_{s,t}^{k}(n)$ extends that of finding Ramsey numbers. Formally,
$$
r_k(s,n) = \min \{ N : f_{k,s}^{k}(N) \ge n\}.
$$

For $k=2$ the above function was first considered by Erd{\H o}s and Rogers~\cite{ERog} only for $t=s+1$, which might be viewed as the most restrictive case. Since then the function has been studied by several researchers culminating in the work of Wolfowitz~\cite{Wo} and Dudek, Retter and R\"odl~\cite{DRR} who proved the upper bound that follows (the lower bound is due to Dudek and the first author~\cite{DM}): for every $s\ge 3$ there are positive constants $c_1$ and $c_2(s)$ such that$$c_1\left(\frac{ N \log N }{\log\log N}\right)^{1/2}< f^2_{s,s+1}(N)< c_2 (\log N)^{4s^2}N^{1/2}.$$
 The problem of estimating the Erd\H os-Rogers function for $k>2$ appears to be much harder. Let us denote $$g(k,N)=f^k_{k+1, k+2}(N)$$ so that the above result (for $s=3$) becomes $g(2,N)=N^{1/2+o(1)}$.
   Dudek and the first author \cite{DM} proved that $(\log N)^{1/4+o(1)} < g(3,N) < O(\log N)$ and more generally that there are positive constants $c_1$ and $c_2$ with
\begin{equation} \label{1} c_1( \log_{(k-2)}N)^{1/4} < g(k,N) < c_2(\log N)^{1/(k-2)}\end{equation}
where $\log_{(i)}$ is the log function iterated $i$ times.
The exponent 1/4 was improved to 1/3 by Conlon, Fox, Sudakov \cite{CFS14}.
Both sets of authors asked whether the upper bound could be improved (presumably  to an iterated log function). Here we prove this where the number of iterations is $k-O(1)$. It remains an open problem to determine the correct number of iterations (which may well be $k-2$).

\begin{theorem}
Fix $k \ge 14$. Then $g(k,N) < O( \log_{(k-13)} N )$.
\end{theorem}

\proof
We will proceed by induction on $k$.  The base case of $k = 14$ follows from the upper bound in (\ref{1}).  For the inductive step, let $k > 14$ and assume that the result holds for $k-1$.  We will show that
$$g(k,2^N)< k \cdot g(k-1, N),$$
and this recurrence  clearly implies the theorem.  Indeed, it easily implies the upper bound
$$g(k,N) < 2^k k! \log_{(k-13)}N$$ by induction on $k$, as $g(k+1, N)$ is at most
$$\begin{array}{ccl}
    g(k+1, 2^{\lceil \log N \rceil} ) & <  & (k+1)g(k, \lceil \log N \rceil) \\\\
      & < & 2^k(k+1)! \log_{(k-13)} \lceil \log N \rceil \\\\
      & \leq &  2^{k+1} (k+1)!\log_{(k-12)} N. \\
  \end{array}$$
Our strategy is to apply a variant of the stepping-up lemma.   Let us begin with a $K_{k+1}^{k-1}$-free $(k-1)$-graph $H'$ on $N$ vertices
for which $\alpha_{k}(H')=g(k-1, N)$. Note that this exists by definition of $g(k-1, N)$.
We will use $H'$  to produce a $K_{k+2}^k$-free $k$-graph $H$ on $2^N$ vertices
with $\alpha_{k+1}(H)< k \alpha_{k}(H')=kg(k-1, N)$.

Let $V(H')=\{0,1,\ldots, N-1\}$ and $V(H) = \{0,1,\ldots, 2^N-1\}$.  For any $a \in V(H)$, write $a=\sum_{i=0}^{N-1}a(i)2^i$ with $a(i) \in \{0,1\}$ for each $i$. For $a \not = b$, let $\delta(a,b)$ denote the largest $i$ for which $a(i) \not = b(i)$.   Therefore Properties A and B in the previous section hold.

Given any set of $s$ vertices $a_1<a_2<\ldots<a_{s}$ of $V(H)$, consider the integers $\delta_i=\delta(a_i,a_{i+1}), 1\le i\le s-1$.  For $e = (a_1,\ldots, a_s)$, let $m(e)$ denote the number of local extrema in the sequence $\delta_1,\ldots, \delta_{s-1}$.  In the case $s = k$, we define the edges of $H$ as follows.  If $\delta_1,\ldots,\delta_{k-1 }$ form a monotone sequence, then let $(a_1,a_2,\ldots,a_{k}) \in E(H)$ if and only if $(\delta_1,\delta_2,\ldots,\delta_{k-1}) \in E(H')$.  Otherwise if $\delta_1,\ldots,\delta_{k-1}$ is not monotone, then $(a_1,a_2,\ldots,a_{k}) \in E(H)$ if and only if $m(e) \in \{k-4, k-3\}$. In other words, given that $\delta_1,\ldots, \delta_{k-1}$ is not monotone, $(a_1,a_2,\ldots,a_{k}) \in E(H)$ if and only if
$\delta_1, \ldots, \delta_{k-1}$ has at most one locally monotone element.
 Note that we have the following variant of Property D.

\begin{description}

\item[Property E:] For $a_1 < \cdots < a_r$, set $\delta_j = \delta(a_j,a_{j + 1})$ and suppose that $\delta_1,\ldots, \delta_{r-1}$ form a monotone sequence.  If every $k$-tuple in $\{a_1,\ldots, a_r\}$ is in $E(H)$ (in $\overline{E}(H)$), then every $(k-1)$-tuple in $\{\delta_1,\ldots, \delta_{r-1}\}$ is in $E(H')$ (in $\overline{E}(H')$).

\end{description}

 We are to show that $H$ contains no $(k+2)$-clique and $\alpha_{k+1}(H)<k\alpha_k(H')$.   First let us establish the following lemma.

 \begin{lemma} \label{l2} Given $e = (a_1,\ldots,  a_7)$ with $a_1<\cdots < a_7$, let $\delta_i = \delta(a_i,a_{i + 1})$ for $1 \leq i \leq 6$.  If $m(e) = 4$, then there is an $a_i$ such that $2 \leq i \leq 6$ and  $m(e-a_i)=2$.
 \end{lemma}

\proof Suppose first that $\delta_2$ is a local minimum, so $\delta_1 > \delta_2 < \delta_3 > \cdots$. Then we have $m(e-a_4) = 2$.  Indeed, since $\delta_4$ is a local minimum, Property B implies $\delta(a_3,a_5) = \delta_3$.  If $\delta_5 > \delta_3$, then we have $\delta_2 < \delta(a_3,a_5) < \delta_5$ and therefore $m(e-a_4) = 2$.  If $\delta_5 < \delta_3$, then we have $\delta(a_3,a_5) > \delta_5 > \delta_6$ which again implies that $m(e-a_4) = 2$.

Now suppose that $\delta_2$ is a local maximum, so $\delta_1< \delta_2 > \delta_3 < \cdots$.  Then we have $m(e-a_3) = 2$.  Indeed, by Property B we have $\delta(a_2,a_4) = \delta_2$.  If $\delta_4 < \delta_2$, then we have $\delta(a_2,a_4) > \delta_4 > \delta_5$ which implies $m(e-a_3) = 2$.  If $\delta_4 > \delta_2$, then we have $\delta_1 < \delta(a_2,a_4) < \delta_4$ which again implies $m(e-a_3) = 2$.  \qed

For sake of contradiction, suppose there are $k+2$ vertices $a_1 < \cdots < a_{k + 2}$ that induce a $K_{k + 2}^k$ in $H$. Define $\delta_i=\delta(a_i, a_{i+1})$ for all $1 \le i \le k+1$.  Given the sequence $\delta_1,\delta_2,\ldots, \delta_{k + 1}$, let us consider the number of locally monotone elements in $D  = \{\delta_2,\ldots, \delta_k\}$.

\emph{Case 1.}  Suppose every element in $D$ is locally monotone.  Then $\delta_1,\ldots, \delta_{k + 1}$ form a monotone sequence.  By Property E, every $(k-1)$-tuple in the set $\{\delta_1,\ldots, \delta_{k + 1}\}$ is an edge in $H'$ which is a contradiction since $H'$ is $K_{k+1}^{k-1}$-free.

\emph{Case 2.}  Suppose there is at least one local extremum $\delta_{\ell}\in D$ and at least
two elements $\delta_{i},\delta_{j} \in D$ that are locally monotone. Then any $k$-tuple $e\subset\{a_1,\ldots, a_{k + 2}\}$ that includes the vertices $$a_{i - 1},a_{i},a_{i + 1}, a_{i + 2}, a_{j - 1},a_{j},a_{j + 1}, a_{j + 2},a_{\ell - 1},a_{\ell},a_{\ell + 1}, a_{\ell + 2}$$ satisfies $1 \leq m(e) < k -4$.  Therefore $e$ is not an edge in $H$ and we have a contradiction.

\emph{Case 3.}  Suppose there is exactly one element $\delta_{i} \in D$ that is locally monotone (and therefore at least one local extremum).  Since $k\geq 15$, either $|\{a_1,\ldots, a_{i - 1}\}| \geq 7$ or $|\{a_{i + 2},\ldots, a_{k + 2}\}| \geq 7$.  Let us only consider the former case, the latter being symmetric.  By Lemma \ref{l2}, there is an element $a_j \in \{a_2,\ldots, a_6\}\subset\{a_1,\ldots, a_{i - 1}\}$ such that for $e' = (a_1,\ldots, a_7)$, $m(e' - a_j) = 2$.  Then any $k$-tuple $e\subset \{a_1,\ldots, a_{k + 2}\}\setminus \{a_j\}$ that includes vertices
$$\{a_t: 1 \le t \le 7, t\ne j\}\cup \{ a_{i - 1}, a_{i},a_{i + 1},a_{i + 2}\}$$ satisfies $1 \leq m(e) < k - 4$.  Hence $e$ is not an edge in $H$ and we have a contradiction.

\emph{Case 4.}  Suppose every element in $D$ is a local extremum.  We then apply Lemma \ref{l2} to the set  $A = \{a_1,\ldots, a_7\}$ and $B = \{a_8,\ldots, a_{14}\}$ to obtain vertices $a_i \in A$ and $a_j \in B$ such that $m(\{a_1, \ldots, a_7\}\setminus\{a_i\})=2$ and $m(\{a_8, \ldots, a_{14}\}\setminus\{a_j\})=2$.
In particular, this implies that
for $e = \{a_1,\ldots, a_{k + 2}\} \setminus \{a_i,a_j\}$, the corresponding sequence of $\delta$'s has at least two locally monotone elements. Since clearly $e$ has at least one local extremum, we obtain $1\le m(e)<k-4$. Hence $e\not\in E(H)$ and we have a contradiction.

Therefore we have shown that $H$ is $K_{k + 2}^k$-free.

Our final task is to show that $\alpha_{k+1}(H)<k\alpha_k(H')$. Set $n = kt$ where $t = \alpha_k(H')$.  Let us assume for contradiction that there are vertices $a_1< \cdots< a_n$ that induce a $(k+1)$-independent set in $H$.  Let $\delta_i = \delta(a_i,a_{i+1})$ for $1\leq i \leq n-1$. If the sequence $\delta_1,\ldots, \delta_{n-1}$ contains fewer than $k$ local extrema, then there is a $j$ such that $\delta_j, \ldots, \delta_{j+t}$ is monotone.  Since $t = \alpha_k(H')$, the $t + 1$ vertices $\{\delta_j,\ldots,\delta_{j + t}\}$ contain a copy of $K^{k-1}_k$ in $H'$.  Say this copy is given by $\delta_{j_1},\ldots, \delta_{j_k}$.  Then by Property E, the vertices $a_{j_1} < \cdots < a_{j_k} < a_{j_k + 1}$ induce a copy of $K_{k + 1}^{k}$ which contradicts our assumption that $\{a_1,\ldots, a_n\}$ is a $(k + 1)$-independent set in $H$.

We may therefore assume that the sequence $\delta_1,\ldots ,\delta_{n-1}$ contains at least $k$ local extrema.  Now we make the following claim.

 \begin{claim}\label{zigzag}
 There is a set of $k + 1$ vertices $a^{\ast}_1,\ldots, a^{\ast}_{k + 1} \in \{a_1,\ldots, a_n\}$ such that for $\delta^{\ast}_i = \delta(a^{\ast}_i,a^{\ast}_{i + 1})$, the sequence $\delta^{\ast}_1,\ldots, \delta^{\ast}_k$ has $k-2$ local extrema.

 \end{claim}

\begin{proof} Let $\delta_{i_1},\ldots, \delta_{i_k}$ be the first $k$ extrema in the sequence $\delta_1,\ldots, \delta_{n-1}$.

\emph{Case 1.}  Suppose $\delta_{i_1}$ is a local minimum.  If $k$ is odd, then consider the $k + 1$ distinct vertices

$$e = a_{i_1},a_{i_1 + 1},a_{i_3},a_{i_3 + 1},a_{i_5},a_{i_5 + 1},\ldots, a_{i_k},a_{i_k + 1}.$$

\noindent Note that the pairs $(a_{i_1},a_{i_1 + 1}),(a_{i_3},a_{i_3 + 1}),(a_{i_5},a_{i_5 + 1}),\ldots $ correspond to local minima.  By Property B, $\delta(a_{i_1 + 1},a_{i_3}) = \delta_{i_2}$, $\delta(a_{i_3 + 1},a_{i_5}) = \delta_{i_4}, \ldots$.  Since $\delta_{i_2}, \delta_{i_4},\delta_{i_6},\ldots$ were local maxima in the sequence $\delta_1,\ldots, \delta_{n-1}$, we have

$$ \delta(a_{i_1},a_{i_1 + 1}) < \delta(a_{i_1 + 1},a_{i_3}) > \delta(a_{i_3},a_{i_3 + 1}) < \delta(a_{i_3 + 1},a_{i_5}) > \cdots.$$  Hence the vertices in $e$ satisfy the claim.  If $k$ is even, then by the same argument as above, the $k+ 1$ vertices

$$a_{1}, a_{i_1},a_{i_1 + 1},a_{i_3},a_{i_3 + 1},a_{i_5},a_{i_5 + 1},\ldots, a_{i_{k-1}},a_{i_{k-1} + 1}$$

\noindent satisfy the claim.

\emph{Case 2.}  Suppose $\delta_{i_1}$ is a local maximum.  If $k$ is odd, then the arguments above imply that the set of $k+1$ vertices

$$a_1,a_2,a_{i_2},a_{i_2 + 1},a_{i_4},a_{i_4 + 1},\ldots, a_{i_{k-1}}, a_{i_{k-1} + 1}$$

\noindent satisfies the claim.  Likewise, if $k$ is even, the set of $k+ 1$ vertices

$$a_1,a_{i_2},a_{i_2 + 1},a_{i_4},a_{i_4 + 1},\ldots, a_{i_{k}}, a_{i_{k} + 1}$$
\noindent satisfies the claim. \end{proof}

By Claim \ref{zigzag}, we obtain $k+1$ vertices $h = (a^{\ast}_1,\ldots, a^{\ast}_{k + 1})$ along with $\delta^{\ast}_1,\ldots , \delta^{\ast}_{k}$ with the desired properties.  Consider the $k$-tuple $e = h - a^{\ast}_i$.  If $i = 1$ or $k + 1$, then it is easy to see that $m(e) = k-3$, which implies $e \in E(H)$.  For $i = 2$, $\delta^{\ast}_3$ is the only possible locally monotone element in the sequence $\delta(a^{\ast}_1,a^{\ast}_3), \delta^{\ast}_3,\ldots, \delta^{\ast}_k$.  Therefore $m(e-a_i) \geq k-4$ and $e\in E(H)$.  A symmetric argument for the case $i = k$ implies that $e \in E(H)$.  Therefore we can assume $3\leq i \leq k-1$.  By Property B, we have $\delta(a^{\ast}_{i-1},a^{\ast}_{i + 1}) = \max\{\delta^{\ast}_{i-1},\delta^{\ast}_i\}$.  Let us consider the two cases.

\emph{Case 1.}  Suppose $\delta(a^{\ast}_{i-1},a^{\ast}_{i + 1}) = \delta^{\ast}_{i-1}$.  If $\delta^{\ast}_{i + 1} > \delta^{\ast}_{i - 1}$, then $\delta^{\ast}_{i-1}$ is the only element in the sequence $\delta^{\ast}_1,\ldots, \delta^{\ast}_{i-1},\delta^{\ast}_{i + 1},\ldots, \delta^{\ast}_k$ that is locally monotone.  Hence $m(e) = k-4$ and $e \in E(H)$.  If $\delta^{\ast}_{i + 1} < \delta^{\ast}_{i-1}$, then $\delta^{\ast}_{i+1}$ is the only possible element in the sequence $\delta^{\ast}_1,\ldots, \delta^{\ast}_{i-1},\delta^{\ast}_{i + 1},\ldots, \delta^{\ast}_k$ that is locally monotone.  More precisely, if $i = k-1$ then $m(e) = k-3$, and if $3 \leq i < k-1$ then $m(e) = k-4$.  Hence $m(e) \geq k-4$ and therefore $e \in E(H)$.

\emph{Case 2.}  Suppose $\delta(a^{\ast}_{i-1},a^{\ast}_{i + 1}) = \delta^{\ast}_{i}$.  If $\delta^{\ast}_{i-2} > \delta^{\ast}_{i}$, then $\delta^{\ast}_{i}$ is the only element in the sequence $\delta^{\ast}_{1},\ldots, \delta^{\ast}_{i-2},\delta^{\ast}_{i},\ldots, \delta^{\ast}_{k}$ that is locally monotone.  Hence $m(e) = k-4$ and $e \in E(H)$.  If $\delta^{\ast}_{i-2} < \delta^{\ast}_{i}$,  then $\delta^{\ast}_{i-2}$ is the only possible element in the sequence $\delta^{\ast}_{1},\ldots, \delta^{\ast}_{i-2},\delta^{\ast}_{i},\ldots, \delta^{\ast}_{k}$ that is locally monotone.  More precisely, if $i = 3$ then $m(e) = k-3$, and if $3 < i \leq k-1$ then $m(e) = k-4$.  Hence $m(e) \geq k-4$ and $e\in E(H)$.

Therefore every $k$-tuple $e = h-a_i$ is an edge in $H$, and the $k + 1$ vertices $h$ induces a $K_{k + 1}^k$ in $H$.  This is a contradiction and we have completed the proof.\qed

\section{Ramsey numbers for $k$-half-graphs versus cliques}\label{last}

Let $K_4^{3}\setminus e$ denote the 3-uniform hypergraph on four vertices, obtained by removing one edge from $K_4^{3}$.  A simple argument of Erd\H os and Hajnal \cite{EH72}  implies $r(K_4^{3}\setminus e,K_n^{3}) < (n!)^2$.    On the other hand, they also gave a construction that shows
$r(K_4^{3}\setminus e,K_n^{3}) > 2^{cn}$ for some constant $c>0$. Improving either of these bounds is a very interesting open problem, as $K_4^{3}\setminus e$ is, in some sense, the smallest 3-uniform hypergraph whose Ramsey number with a clique is at least exponential.

A \emph{$k$-half-graph}, denote by $B^k$, is a $k$-uniform hypergraph on $2k-2$ vertices, whose vertex set is of the form $S\cup T$, where $|S| = |T| = k-1$, and whose edges are all $k$-subsets that contain $S$, and one $k$-subset that contains $T$.  The hypergraph $B^k$ can be viewed as a generalization of $K_4^{3}\setminus e$ as $B^{3}=K_4^{3}\setminus e$.

The goal of this section is to obtain upper and lower bounds for $r(B^k,K^k_n)$ that parallel the known state of affairs for $K_4^{3}\setminus e$.
We begin by presenting a straightforward generalization of the argument of Erd\H os and Hajnal to establish an upper bound for Ramsey numbers for $k$-half-graphs versus cliques. Again for simplicity we write $r(B^k,K^k_n)  = r_k(B,n)$.

\begin{theorem}\label{halfup}
For $k\geq 4$, we have $ r_k(B,n) \leq (n!)^{k-1}.$

\end{theorem}

\noindent First, let us recall an old lemma due to Spencer.

\begin{lemma}[\cite{S}]\label{spencer}
Let $H = (V,E)$ be a $k$-uniform hypergraph on $N$ vertices.  If $|E(H)| > N/k$, then there exists a subset $S\subset V(H)$ such that $S$ is an independent set and

$$|S| \geq \left(1 - \frac{1}{k}\right)N \left(\frac{N}{k|E(H)|}\right)^{\frac{1}{k-1}}.$$

\end{lemma}

\medskip

\noindent \emph{Proof of Theorem \ref{halfup}.}  We proceed by induction on $n$.  The base case $n = k$ is trivial.  Let $n > k$ and assume the statement holds for $n'  < n$.  Let $k\geq 4$ and let $\chi$ be a red/blue coloring on the edges of $K^k_N$, where $N = (n!)^{k-1}$.  Let $E_R$ denote the set of red edges in $K^k_N$.

\medskip

\noindent \emph{Case 1:} Suppose $|E_R|\leq N/k$.  Then one can delete $N/k$ vertices from $H$ and obtain a blue clique of size $(1 - 1/k)N \geq n$.

\medskip

\noindent \emph{Case 2}:  Suppose $N/k < |E_R| < \frac{\left(1 - \frac{1}{k}\right)^{k-1}N^k}{ k n^{k-1}  }.$ Then by Lemma \ref{spencer}, $K^k_N$ contains a blue clique of size $n$.

\medskip

\noindent \emph{Case 3}:  Suppose $|E_R| \geq \frac{\left(1 - \frac{1}{k}\right)^{k-1}N^k}{ k n^{k-1}  }.$  Then by averaging, there is a $(k-1)$-element subset $S\subset V$ such that $N(S) = \{v \in V: S\cup \{v\} \in E_R\}$ satisfies

$$|N(S)| \geq \frac{\left(1 - \frac{1}{k}\right)^{k-1}N^k}{ n^{k-1}{N\choose k-1} }  \geq \left((n-1)!\right)^{k-1} .$$

\noindent  The last inequality follows from the fact that $k\geq 4$.  Fix a vertex $u \in S$.  If $\{u\}\cup T \in E_R$ for some $T\subset N(S)$ such that $|T| = k-1$, then $S\cup T$ forms a red $B^k$ and we are done.  Therefore we can assume otherwise.  By the induction hypothesis, $N(S)$ contains a red copy of $B^k$, or a blue copy of $K^k_{n-1}$.  We are done in the former case, and in the latter case, we can form a blue $K_{n}^k$ by adding the vertex $u$. $\hfill\square$

\medskip

We now move to our main new contribution, which are constructions which show that $r_k(B,n)$ is at least exponential in $n$.

\begin{theorem}\label{halflow}
For fixed $k\geq 3$, we have $ r_k(B,n) > 2^{\Omega(n)}.$

\end{theorem}

\begin{proof} Surprisingly, we require different arguments for $k$ even and $k$ odd.

\medskip

\noindent \emph{The case when $k$ is odd.}  Assume $k$ is odd, and set $N = 2^{cn}$ where $c = c(k)$ will be determined later.  Then let $T$ be a random tournament on the vertex set $ [N]$, that is, for $i,j \in [N]$, independently, either $(i,j) \in E$ or $(j,i) \in E$, where each of the two choices is equally likely.  Then let $\chi:{[N]\choose k} \rightarrow \{\textnormal{red},\textnormal{blue}\}$ be a red/blue coloring on the $k$-subsets of $[N]$, where $\chi(v_1,\ldots, v_k) =$ red if $v_1,\ldots, v_k$ induces a \emph{regular} tournament, that is, the indegree of every vertex is $(k-1)/2$ (and hence the outdegree of every vertex is $(k-1)/2$).  Otherwise we color it blue.  We note that since $k$ is odd, a regular tournament on $k$ vertices is possible by the fact that $K_{k}$ has an Eulerian circuit, and then by directing the edges according to the circuit we obtain a regular tournament.

Notice that the coloring $\chi$ does not contain a red $B^k$.  Indeed, let $S,T \subset [N]$ such that $|S| = |T| = k-1$, $S\cap T = \emptyset$, and every $k$-tuple of the form $S\cup \{v\}$ is red, for all $v \in T$.  Then for any $u \in S$, all edges in the set $u\times T$ must have the same direction, either all emanating out of $u$ or all directed towards $u$.  Therefore it is impossible for $u\cup T$ to have color red, for any choice $u \in S$.

Next we estimate the expected number of monochromatic blue copies of $K^k_n$ in $\chi$.  For a given $k$-tuple $v_1,\ldots,v_k \in [N]$, the probability that $\chi(v_1,\ldots,v_k) = \textnormal{blue}$ is clearly at most $1 - 1/2^{{k\choose 2}}$.  Let $T = \{v_1,\ldots, v_n\}$ be a set of $t$ vertices in $[n]$, where $v_1 < \cdots < v_n$.  Let $S$ be a partial Steiner $(n,k,2)$-system with vertex set $T$, that is, $S$ is a $k$-uniform hypergraph such that each $2$-element set of vertices is contained in at most one edge in $S$.  Moreover, $S$ satisfies $|S| = c'n^{2}$ where $c'  = c'(k)$.  It is known that such a system exists. Then the probability that every $k$-tuple in $T$ has color blue is at most the probability that every $k$-tuple in $S$ is blue.  Since the edges in $S$ are independent, that is no two edges have more than one vertex in common, the probability that $T$ is a monochromatic blue clique is at most $\left(1 - 1/2^{{k\choose 2}}\right)^{|S|} \leq \left(1 - 1/2^{{k\choose 2}}\right)^{c'n^{2}}$.  Therefore the expected number of monochromatic blue copies of $K^k_n$ in $\chi$ is at most

$${N\choose n}\left(1 - 1/2^{{k\choose 2}}\right)^{c'n^{2}} < 1,$$

\noindent for an appropriate choice for $c = c(k)$.  Hence, there is a coloring $\chi$ with no red $B^k$ and no blue $K^k_n$.  Therefore

$$r_k(B,n) > 2^{cn }.$$

\medskip

\noindent \emph{The case when $k$ is even.}   Assume $k$ is even and set $N = 2^{cn}$ where $c = c(k)$ will be determined later.  Consider the coloring $\phi:{[N]\choose 2} \rightarrow \{1,\ldots, k-1\}$, where each edge has probability $1/(k-1)$ of being a particular color independent of all other edges (pairs).  Using $\phi$, we define the coloring $\chi:{[N]\choose k}\rightarrow \{\textnormal{red},\textnormal{blue}\}$, where the $k$-tuple $(v_1,\ldots, v_k)$ is red if $\phi$ is a proper edge-coloring on all pairs among $\{v_1,\ldots, v_k\}$, that is, each of the $k-1$ colors appears as a perfect matching.  Otherwise we color it blue.

Notice that the coloring $\chi$ does not contain a red $B^k$.  Indeed let $S,T \subset [N]$ such that $|S| = |T| = k-1$ and $S\cap T = \emptyset$.  If, for all $v \in T$, the $k$-tuples of the form $S\cup \{v\}$ are red, then the set of edges $\{u\}\times T$ is monochromatic with respect to $\phi$ for any $u\in S$.  Hence, $\chi$ could not have colored $\{u\}\cup T$ red for any $u \in S$.

For a given $k$-tuple $v_1,\ldots,v_k \in [N]$, the probability that $\chi(v_1,\ldots,v_k) = \textnormal{blue}$ is at most $1 - (1/(k-1))^{{k\choose 2}}$.  By the same argument as above, the expected number of monochromatic blue copies of $K^k_n$ with respect to $\chi$ is less than 1 for an appropriate choice of $c = c(k)$.  Hence, there is a coloring $\chi$ with no red $B^k$ and no blue $K^k_n$.  Therefore

$$r_k(B,n) > 2^{cn}$$
and the proof is complete.
\end{proof}

{\bf Acknowledgment.} We thank the referee for helpful comments.


\begin{thebibliography}{99}

    \bibitem{AKS} M. Ajtai, J. Koml\'os, E. Szemer\'edi, A note on Ramsey numbers, {\it J. Combin. Theory Ser. A} {\bf 29} (1980), 354--360.


\bibitem{B} T. Bohman, The triangle-free process, {\it Adv. Math.} {\bf 221} (2009), 1653--1677.

\bibitem{BK} T. Bohman, P. Keevash, The early evolution of the $H$-free process, {\it Invent. Math.} {\bf 181} (2010), 291--336.



\bibitem{CFS15} D. Conlon, J. Fox, B. Sudakov, personal communication.

\bibitem{CFS13} D. Conlon, J. Fox, and B. Sudakov, An improved bound for the stepping-up lemma, \emph{Discrete Applied Mathematics} \textbf{161} (2013), 1191--1196.

\bibitem{CFS} D. Conlon, J. Fox, and B. Sudakov, Hypergraph Ramsey numbers, \emph{J. Amer. Math. Soc.} \textbf{23} (2010), 247--266.

\bibitem{CFS14} D. Conlon, J. Fox, and B. Sudakov, Short proofs of some extremal results, \emph{Combin. Probab.
Comput.} \textbf{23} (2014), 8--28.


\bibitem{DM} A. Dudek, D. Mubayi, On generalized Ramsey numbers for 3-uniform hypergraphs, \emph{J. Graph Theory} \textbf{76} (2014), 217--223.

\bibitem{DRR} A. Dudek, T. Retter, V. R\"odl, On generalized Ramsey numbers of Erd\H os and Rogers, {\em J. Combin. Theory Ser. B} 109 (2014), 213--227.


\bibitem{FPSS} J. Fox, J. Pach, B. Sudakov, A. Suk, Erd\H os-Szekeres-type theorems for monotone paths and convex bodies, \emph{Proc. Lond. Math. Soc.} \textbf{105} (2012), 953--982.



\bibitem{E47} P. Erd\H os, Some remarks on the theory of graphs, \emph{Bull. Amer. Math. Soc.} \textbf{53} (1947), 292--294.



\bibitem{EH72} P. Erd\H os, A. Hajnal, On Ramsey like theorems, problems and results, in Combinatorics (Proc. Conf. Combinatorial Math., Math. Inst., Oxford, 1972), pp. 123--140, Inst. Math. Appl., Southhend-on-Sea, 1972.

\bibitem{EHR} P. Erd\H os, A. Hajnal, R. Rado, Partition relations for cardinal numbers, \emph{Acta Math. Acad.
Sci. Hungar.} \textbf{16} (1965), 93--196.

\bibitem{ER} P. Erd\H os, R. Rado, Combinatorial theorems on classifications of subsets of a given set, \emph{Proc. Lond. Math. Soc.} \textbf{3} (1952), 417--439.


\bibitem{ERog} P.~Erd\H{o}s, C.A.~Rogers,
{\em The construction of certain graphs},
\emph{Canad. J.~Math.}~\textbf{14} (1962), 702--707.

\bibitem{ES35} P. Erd\H os, G. Szekeres, A combinatorial problem in geometry, \emph{Compos. Math.} \textbf{2} (1935), 463--470.





\bibitem{graham} R. L. Graham, B. L.  Rothschild, J. H. Spencer: {\em Ramsey Theory, 2nd ed.}, Wiley, New York, 1990.



\bibitem{MS15} D. Mubayi, A. Suk, Off-diagonal hypergraph Ramsey numbers, {\em J. Combin. Theory Ser. B} \textbf{125} (2017), 168--177.



\bibitem{S} J. Spencer, Tur\'an theorem for $k$-graphs, {\em Disc. Math.} \textbf{2} (1972), 183--186.





\bibitem{Wo} G. Wolfowitz, $K_4$-free graphs without large induced triangle-free subgraphs, {\em Combinatorica}
 \textbf{33} (2013), no. 5, 623--631.

 \end{thebibliography}
\end{document}